\documentclass[preprint,12pt]{elsarticle}

\usepackage{amssymb}
\usepackage{amsthm}
\usepackage{amsmath}
\usepackage{epic}

\biboptions{sort&compress}

\newtheorem{thm}{Theorem}[section]
	\newtheorem{lem}[thm]{Lemma}
    
    \newtheorem{cor}[thm]{Corollary}
        \newtheorem{claim}{Claim}[thm]
\newtheorem{defi}{Definition}

\journal{}

\begin{document}

\begin{frontmatter}

\title{The high order spectral radius of graphs without long cycles or paths }

\author{Yuntian Wang, Lizhu Sun, Changjiang Bu}

\address{College of Mathematical Sciences, Harbin Engineering University, Harbin 150001, PR China}

\begin{abstract}
    In 1959, Erd\H{o}s and Gallai established two classic theorems, which determine the maximum number of edges in an $n$-vertex graph with no cycles of length at least $k$, and in an $n$-vertex graph with no paths on $k$ vertices, respectively.
    Subsequently, generalized and spectral versions of the Erd\H{o}s-Gallai theorems have been investigated.
    A concept of a high order spectral radius for graphs was introduced in 2023, defined as the spectral radius of a tensor and termed the $t$-clique spectral radius $\rho_t(G)$.
    In this paper, we establish a high order spectral version of Erd\H{o}s-Gallai theorems by employing the $t$-clique spectral radius, i.e., we determine the extremal graphs that attain the maximum $t$-clique spectral radius in the $n$-vertex graphs with no cycles of length at least $k$ and in the $n$-vertex graphs with no paths on $k$ vertices, respectively.
\end{abstract}

\begin{keyword}
High order spectral radius, Spectral extremal, Tensor, Cycles, Paths.

\end{keyword}

\end{frontmatter}

\section{Introduction}

        Let $G$ be a graph with vertex set $V(G)$ and edge set $E(G)$.
        For a vertex $v \in V(G)$, let $N_G(v)$ be the set of neighbors of $v$ in $G$, $d_G(v)=|N_G(v)|$ be the degree of $v$ in $G$.
        The length of a longest cycle in $G$ is its circumference, denoted by $c(G)$, and the number of vertices in a longest path in $G$ is denoted by $p(G)$.
        For two vertices $u, v \in V(G)$, we write $u \sim v$ (respectively, $u \nsim v$) to indicate that $u$ and $v$ are adjacent (respectively, non-adjacent).
        Let $K_n$ denote the complete graph with $n$ vertices, and $I_n$ denote the empty graph with $n$ vertices.
        Let $K_t(G)$ denote the set of all cliques (i.e., vertex sets of complete subgraphs ) of size $t$ in $G$.
        Let $C_{\geq k}$ denote the family of all cycles of length at least $k$, and let $P_k$ denote a path on $k$ vertices.
        For two vertices $u, v \in V(G)$, the distance  between $u$ and $v$ of $G$ is the length of the shortest path connecting them. The diameter of $G$ is the maximum distance among all vertices of $G$.
        For any two disjoint graphs $G$ and $H$, define their union as $G \cup H$, with vertex set $V(G) \cup V(H)$ and edge set $E(G) \cup E(H)$, their join is defined as $G \vee H$, with vertex set $V(G) \cup V(H)$ and edge set $E(G) \cup E(H) \cup \{\{u,v\} \mid u \in V(G), v \in V(H)\}$.
        For positive integers $n, l, a$ with $a \leq n-l$, define the $n$-vertex graph $S_{n,l,a} = K_l \vee (K_a \cup I_{n-l-a})$. For brevity, we denote $S_{n,l,1}$ by $S_{n,l}$ and $S_{n,l,2}$ by $S_{n,l}^+$.

        For a given graph family $\mathcal{F}$, a graph $G $ is called $\mathcal{F} $-free if it does not contain any member of $\mathcal{F} $ as a subgraph. When $\mathcal{F}$ consists of a single graph $F$, i.e. $\mathcal{F}=\{F\}$, we write $F$-free as a shorthand for $\{F\}$-free.

        In 1959, Erd\H{o}s and Gallai \cite{Erdos1959} determined the maximum number of edges for two types of $n$-vertex graphs: those with no cycles of length at least $k$, and those with no paths on $k$ vertices, as stated in the following two theorems:
        \begin{thm}\label{erdos1}\cite{Erdos1959}
            Let $k \geq 3$. For an $n$-vertex graph $G$ that is $C_{\geq k}$-free, the number of edges satisfies
                \begin{align*}
                    e(G) \leq \frac{1}{2}(k-1)(n-1).
                \end{align*}
		\end{thm}
        \begin{thm}\label{erdos2}\cite{Erdos1959}
            Let $k \geq 2 $. For an $n$-vertex graph $G$ that is $P_{k}$-free, the number of edges satisfies
                \begin{align*}
                    e(G) \leq \frac{1}{2}(k-2)n.
                \end{align*}
		\end{thm}

        The bounds in the above two theorems are sharp, with extremal examples arising when the specified divisibility conditions hold: for Theorem~\ref{erdos1}, take any connected $n$-vertex graph whose blocks (maximal connected subgraphs with no cut vertices) are $K_{k-1} $  when $k-2$ divides $n-1$; and for Theorem~\ref{erdos2}, take the graph whose connected components are  $K_{k-1} $ when $k-1$ divides $n-1$.

        For a positive integer $s$, let $N_s(G)$ denote the number of $s$-cliques in graph $G$. A generalized version (clique version) of Theorem~\ref{erdos1} was established by Luo \cite{Luo2018}. First, Luo derived an upper bound for the maximum number of $t$-cliques in $C_{\geq k}$-free 2-connected graphs (connected graphs without cut vertices) on $n$ vertices (Theorem~\ref{CT2-2}). Then, using this result, Luo obtained the clique version of Theorem \ref{erdos1} (Corollary~\ref{CT2-3}),
        \begin{thm}\label{CT2-2}\cite{Luo2018}
            Let $n \geq k \geq 5$, let $G$ be a $C_{\geq{k}}$-free 2-connected graph on $n$ vertices.

            \noindent (1)If $k$ is odd, $N_t(G) \leq \max\{N_t(S_{n,2,k-4}), N_t(S_{n, \frac{k-1}{2} })\}$;

            \noindent (2)If $k$ is even, $N_t(G) \leq \max\{N_t(S_{n,2,k-4}), N_t(S_{n, \frac{k-2}{2} }^+)\}$.

		\end{thm}
        \begin{cor}\label{CT2-3}\cite{Luo2018}
            Let $n \geq k \geq 4$, let $G$ be a $C_{\geq{k}}$-free graph on $n$ vertices, then
            $N_t(G) \leq \frac{n-1}{k-2}\begin{pmatrix}k-1\\t\end{pmatrix}$.
		\end{cor}

        Similarly, a generalized version (clique version) of Theorem~\ref{erdos2} has also been given by Luo \cite{Luo2018}. First, Luo established an upper bound on the maximum number of $t$-cliques in connected graphs with $n$ vertices that do not contain $P_k$ as a subgraph (Theorem~\ref{CT2-4}). Furthermore, Luo derived the clique version of Theorem~\ref{erdos2} (Corollary~\ref{CT2-5}).
        \begin{thm}\label{CT2-4}\cite{Luo2018}
            Let $n \geq k \geq 4$, let $G$ be a $P_k$-free connected graph on $n$ vertices.

            \noindent (1)If $k$ is odd, $N_t(G) \leq \max\{N_t(S_{n,1,k-3}), N_t(S_{n, \frac{k-3}{2} }^+)\}$;

            \noindent (2)If $k$ is even, $N_t(G) \leq \max\{N_t(S_{n,1,k-3}), N_t(S_{n, \frac{k-2}{2} })\}$.

		\end{thm}
        \begin{cor}\label{CT2-5}\cite{Luo2018}
            Let $G$ be a $P_{{k}}$-free graph on $n$ vertices, then
            $N_t(G) \leq \frac{n}{k-1}\begin{pmatrix}k-1\\t\end{pmatrix}$.
		\end{cor}

        Recent work \cite{Zhao2024, Lu2025, Fang2025} has studied the extremal problem of the maximum number of $t$-cliques in graphs that forbid long cycles and other subgraphs, as well as in graphs that forbid long paths and other subgraphs.

        Let $\mu(G)$ denote the spectral radius of the adjacency matrix of graph $G$. The spectral versions of Erd\H{o}s and Gallai theorems were established by Gao and Hou \cite{Gao2019} for graphs with no cycles of length at least $k$, and by Nikiforov \cite{Nikiforov2010} for graphs with no paths on $k$ vertices.

        \begin{thm}\label{CT2}\cite{Gao2019}
            For positive integers $k$ and $n$ with $s = \lfloor \frac{k-1}{2} \rfloor \geq 2$ and $n \geq 13{s}^2$, let $G$ be a $C_{\geq k}$-free graph on $n$ vertices.

            \noindent (1)If $k $ is odd, $\mu(G) \leq \mu(S_{n,s})$, equality holds if and only if $G = S_{n,s}$;

            \noindent (1)If $k $ is even, $\mu(G) \leq \mu(S_{n,s}^+)$, equality holds if and only if $G = S_{n,s}^+$.
		\end{thm}

        \begin{thm}\label{CT2-1}\cite{Nikiforov2010}
            For positive integers $k$ and $n$ with $s = \lfloor \frac{k-1}{2} \rfloor \geq 2$ and $n \geq 2^{4s}$, let $G$ be a $P_k$-free graph on $n$ vertices.

            \noindent (1)If $k $ is odd, $\mu(G) \leq \mu(S_{n,s-1}^+)$, equality holds if and only if $G = S_{n,s-1}^+$;

            \noindent (1)If $k $ is even, $\mu(G) \leq \mu(S_{n,s})$, equality holds if and only if $G = S_{n,s}$.
		\end{thm}

        For other spectral results concerning cycles and paths, see~\cite{Zhai2015,Lin2021,Zhai2021,Li2022,Zhang2023,Li2023, Li2024,C2024,Zou2026}. The concept of tensor eigenvalues was independently proposed by Qi \cite{Qi2005} and Lim \cite{Lim2005} in 2005.

        \begin{defi}\cite{Qi2005,Lim2005}
            Let $H = (h_{i_1i_2\cdots i_t})$ be a $t$-order $n$-dimensional complex tensor, where $i_j=1,\cdots,n$ and $j=1,\cdots, t$. A complex number $\lambda$ is called an eigenvalue of $H $ if there exists a complex vector $x=(x_1,x_2,\dots,x_n)^\textrm{T}\neq 0$ satisfing
                \begin{align*}
				    Hx^{t-1} = \lambda x^{[t-1]},
			    \end{align*}
            and $x$ is called an eigenvector corresponding to $\lambda$,
            where $Hx^{t-1}$ is a vector in $\mathbb{C}^n$ whose $i$-th component is
                \begin{equation}\label{align1}
                    (Hx^{t-1})_i=\sum\limits_{i_2,\ldots,i_t=1}^n h_{ii_2\cdots i_t}x_{i_2}\cdots x_{i_t},
                \end{equation}
            and $x^{[t-1]}=(x_1^{t-1},x_2^{t-1},\dots,x_n^{t-1})^\textrm{T}$. The spectral radius $\rho(H)$ represents the maximum value of the modulus of all eigenvalues.
        \end{defi}

        The following introduces a high order spectral radius for graphs, based on the definition of tensor eigenvalues, which is referred to as the $t$-clique spectral radius, as proposed by Liu and Bu \cite{Liu2023}.
        \begin{defi}\cite{Liu2023}
            let $G$ be a graph on $n$ vertices, the tensor $A_t(G) = (a_{i_1i_2\cdots i_t})$ of $t$-order $n$-dimension is called the $t$-clique tensor of $G$, where
			\begin{align*}
                a_{i_1i_2\cdots i_t}=\begin{cases}\frac{1}{(t-1)!},&\text{if }\{i_1,\ldots,i_t\}\in K_t(G).\\0,&\text{otherwise}.\end{cases}
			\end{align*}
            The spectral radius of $A_t(G)$ is denoted by $\rho_t(G)$, which is called the $t$-clique spectral radius of $G$.
        \end{defi}

        When $t=2$, $\rho_t(G)$ is the adjacency spectral radius of the graph $G$.

        The extension of classical extremal graph theory results to the $t$-clique spectral setting has seen notable progress. Liu and Bu were among the first to advance this line of research, by providing clique spectral versions of Mantel's theorem \cite{Liu2023} and the Erd\H{o}s-Simonovits stability theorem \cite{Liu2024}. Liu, Zhou and Bu \cite{Liu2025} established a spectral analogue of the Erd\H{o}s-Stone theorem. Meanwhile, Tur\'{a}n-type problems for the clique spectral radius have been actively studied. Liu, Zhou and Bu \cite{Liu20231} derived an upper bound for the $3$-clique spectral radius of graphs forbidding the book graph $B_k$ and the complete bipartite graph $K_{2,l}$. Yu and Peng \cite{Yu2026} then characterized the extremal graphs for the $t$-clique spectral radius of graphs excluding a matching of size $t$ ($M_t$). Subsequently, Yan, Yang, and Peng \cite{Yan2026} solved the problem for graphs forbidding both $K_{k}$ and $M_{s}$.

        In this paper, we consider a $t$-clique spectral version of the Erd\H{o}s-Gallai theorem, presented as Theorem~\ref{erdos1} and Theorem~\ref{erdos2}. This work extends Theorem~\ref{CT2} and Theorem~\ref{CT2-1} to the clique spectral version, respectively, at the same time, these results we obtained are the spectral versions of Corollary~\ref{CT2-3} and Corollary~\ref{CT2-5}, respectively.

        \begin{thm}\label{CT1-1}
            For $n \geq k-1$ and $2 \leq \lfloor \frac{k+1}{2} \rfloor < t < k$, let $G$ be a $C_{\geq{k}}$-free graph on $n$ vertices, then $\rho_t(G) \leq \rho_t(K_{k-1})$.

        \end{thm}

        We will introduce the concept of $t $-clique connected component in the next section.

        \begin{thm}\label{CT1}
            For $2 \leq t \leq \lfloor \frac{k+1}{2} \rfloor$ and sufficiently large $n$, let $G$ be a $C_{\geq{k}}$-free graph on $n$ vertices.

            \noindent (1)If $k$ is odd, $\rho_t(G) \leq \rho_t(S_{n, \frac{k-1}{2} })$, equality holds if and only if $G = S_{n, \frac{k-1}{2} }$;

            \noindent (2)If $k$ is even, $\rho_t(G) \leq \rho_t(S_{n, \frac{k-2}{2} }^+)$, equality holds if and only if $G = S_{n, \frac{k-2}{2} }^+$.

		\end{thm}

        \begin{thm}\label{CT3-1}
            For $n \geq k-1$ and $2 \leq \lfloor \frac{k+1}{2} \rfloor < t < k $, let $G$ be a $P_{k}$-free graph on $n$ vertices, then $\rho_t(G) \leq \rho_t(K_{k-1})$.

        \end{thm}

        \begin{thm}\label{CT3}
            For $2 \leq t \leq \lfloor \frac{k+1}{2} \rfloor$ and sufficiently large $n$, let $G$ be a $P_{k}$-free graph on $n$ vertices.

            \noindent (1)If $k$ is odd, $\rho_t(G) \leq \rho_t(S_{n, \frac{k-3}{2} }^+)$, equality holds if and only if $G = S_{n, \frac{k-3}{2} }^+$;

            \noindent (2)If $k$ is even, $\rho_t(G) \leq \rho_t(S_{n, \frac{k-2}{2} })$, equality holds if and only if $G = S_{n, \frac{k-2}{2} }$.

		\end{thm}

\section{Preliminaries}

        In this section, we introduce some concepts and lemmas.
        Let $A = (a_{i_1i_2\cdots i_t})$ be a $t$-order $n$-dimensional real tensor, if $a_{i_1 i_2 \cdots i_t} \geq 0 $ for $i_j = 1, 2, \cdots, n$ and $j = 1, 2, \cdots, t$, then $A $ is called nonnegative; for any permutation $\sigma$ of $i_1 i_2 \cdots i_t$, if $a_{i_1 i_2 \cdots i_t} = a_{\sigma(i_1 i_2 \cdots i_t)}$ for $i_j = 1, 2, \cdots, n$ and $j = 1, 2, \cdots, t$, then $A $ is called symmetric.

        \begin{lem}\label{1L1}\cite{Qi2013}
            For a nonnegative symmetric tensor $A $ of $t$-order and $n$-dimension, let the nonnegative vector $x = (x_1, x_2, \dots, x_n)^\textrm{T}$, then
            \begin{align*}
                \rho(A) = \max\left\{ x^\textrm{T}Ax^{t-1}\mid \sum_{i=1}^{n} x_i^t = 1 \right\}.
            \end{align*}
		\end{lem}

        Clearly, the $t$-clique tensor is nonnegative and symmetric, for an $n$-vertex graph $G$, Lemma~\ref{1L1} implies that $\rho_t(G)$ satisfies
            \begin{align*}
                \rho_t(G)= \max\left\{ A_t(G)x^t \mid x = (x_1, x_2, \dots, x_n)^\textrm{T} \geq 0, \sum_{i=1}^{n} x_i^t = 1 \right\},
            \end{align*}
        where $A_t(G)x^t := x^\textrm{T}A_t(G)x^{t-1}$.

        Let $A=(a_{i_1i_2 \cdots i_t})$ be a tensor of $t$-order and $n$-dimension, if there exists a nonempty proper index subset $I \subset [n]$ such that
			\begin{align*}
                a_{i_1i_2 \cdots i_t}=0,\ for\ all\ i_1 \in\ I\ and\ \{i_1i_2 \cdots i_t\}\not\subseteq\ I,
			\end{align*}
        then $A$ is called a weakly reducible tensor. Otherwise, $A$ is called a weakly irreducible tensor.

        The two vertices $i$, $j$ on a graph $G$ are called $t$-clique connected, if there are some $t$ cliques $K_t^1, K_t^2,\cdots , K_t^s \in K_t(G)$ such that $i \in V(K_t^1)$, $j \in V(K_t^s)$ and $V(K_t^k) \cap V(K_t^{k+1}) \neq \emptyset$ for $k=1, \cdots, s-1$. If any two vertices of a graph $G$ are $t$-clique connected, then the graph $G$ is called $t$-clique connected \cite{Liu2023}, the maximal $t$-clique connected subgraph in $G$ is called a $t$-clique connected component of $G$.

        \begin{lem}\label{2L1}\cite{Liu2023}
            The $t$-clique tensor of a graph $G$ is weakly irreducible if and only if $G$ is $t$-clique connected.
        \end{lem}

        \begin{lem}\label{2L3}\cite{Friedland2013}
            If the tensor $A$ is nonnegative weakly irreducible, then $\rho(A)>0$ is the eigenvalue of $A$, and it corresponds to a unique positive eigenvector of $A$ up to a multiplicative constant.
        \end{lem}

        From Lemma~\ref{2L1} and Lemma~\ref{2L3}, it follows that the $t$-clique spectral radius of a $t$-clique connected graph $G$ corresponds to a unique positive eigenvector $x=\{x_1, x_2, \cdots, x_n\}$ satisfying $\sum_{i=1}^{n} x_i^t = 1$, which we refer to as the $t$-clique Perron vector of graph $G$.

        \begin{lem}\label{2L2}\cite{Shao2013}
            For a weakly reducible tensor $A$, there is a lower triangular block tensor such that $A$ is permutational similar to it and each diagonal block tensor is weakly irreducible. Furthermore, there is a diagonal block tensor such that its spectral radius is equal to $\rho(A)$.
        \end{lem}

        The following Lemma can be directly derived from Lemma~\ref{2L1} and Lemma~\ref{2L2}.
        \begin{lem}\label{2L2-1}
            If the graph $G$ is not $t$-clique connected, then there exists a $t$-clique connected component $H$ of $G$ such that $\rho_t(H) = \rho_t(G)$.
        \end{lem}

        \begin{lem}\label{2L4}\cite{Khan2015}
            For two unequal $t$-order $n$-dimensional nonnegative tensors $B$, $C$, if $B-C$ is nonnegative and $B$ is weakly irreducible, then $\rho(C)<\rho(B)$.
        \end{lem}

        For a graph $G$, note $u, v$ is a pair of vertices in graph $G$, defined
            \begin{align*}
                P_v(u)=\{i\mid i\in N_{G}(u)\setminus\{v\},\mathrm{~and~}i\notin N_{G}(v)\}.
            \end{align*}
        Below is an introduction to a graph operation on $G $. Let $G_{u \to v}$ be the graph obtained from $G$ by deleting the edges between $u$ and the vertices in $P_v(u)$ and adding new edges connecting $v$ to vertices in $P_v(u)$, that is $V(G_{u \to v})=V(G)$ and
            \begin{align*}
                E(G_{u \to v}) = E(G) \setminus \{ \{u, i\} \mid i \in P_v(u) \} \cup \{ \{v, i\} \mid i \in P_v(u) \}.
            \end{align*}

        In 2019, Gao and Hou \cite{Gao2019} showed that such an operation, provided it is only applied to pairs of adjacent vertices, will terminate in finite steps and gave the following lemma.

        \begin{lem}\label{Gao1}\cite{Gao2019}
            Let $G$ be a graph with $n$ vertices and $\{u,v\} \in E(G)$, then $c(G_{u \to v}) \leq c(G)$.
        \end{lem}
        In 2024, Ai, Lei and Ning gave the following lemma.
        \begin{lem}\label{L4-1}\cite{Ai2024}
            Let $G$ be a graph with $n$ vertices and $\{u,v\} \in E(G)$, then $p(G_{u \to v})\leq p(G)$.
        \end{lem}

\section{Proof of Theorems~\ref{CT1-1} and \ref{CT1}}

        In this section, we prove Theorems~\ref{CT1-1} and \ref{CT1}. We first provide the following lemma.

        \begin{lem}\label{CL1}
                Let $G$ be a graph with $n$ vertices, and let $u,v \in V(G)$, then $\rho_t(G_{u \to v}) \geq \rho_t(G)$.
        \end{lem}
        \begin{proof}
                If the graph $G$ is $t$-clique connected, by Lemma~\ref{2L1}, the tensor $A_t(G)$ is weakly irreducible. Let $x$ be the $t$-clique Perron vector of graph $G$ with the components corresponding to vertices $u$ and $v$ denoted by $x_u$ and $x_v$ respectively. Since $G_{u \to v}$ is isomorphic to $G_{v \to u}$, without loss of generality, we may assume $x_u \leq x_v$, otherwise, we can consider $G_{v \to u}$ as an alternative. We can construct an injection $\phi$ from the set of $t$-cliques in $G$ to the set of $t$-cliques in $G_{u\to v}$, such that for any $t$-clique $K_t^0$ in $G$, either $V(\phi(K_t^0))=V(K_t^0)$ or $V(\phi(K_t^0)) = V(K_t^0)\setminus \{u\} \cup \{v\}$, due to the condition $x_u \leq x_v$, we have
                    \begin{align*}
                        \sum_{\{i_1,\cdots ,i_t\} = \phi(K_t^0)}x_{i_1}x_{i_2}\cdots x_{i_t} \geq \sum_{\{i_1,\cdots ,i_t\}=  K_t^0}x_{i_1}x_{i_2}\cdots x_{i_t},
			        \end{align*}
                furthermore, there is
                    \begin{align*}
                         &A_t(G_{u\to v})x^t - A_t(G)x^t\\
                         = & t\sum_{\{i_1,\cdots ,i_t\}\in K_t(G_{u\to v})}x_{i_1}x_{i_2}\cdots x_{i_t}-t\sum_{\{i_1,\cdots ,i_t\}\in K_t(G)}x_{i_1}x_{i_2}\cdots x_{i_t}\\
                         \geq & 0.
			        \end{align*}

                Then, we have
                    \begin{align}\label{align1}
                         \rho_t(G_{u\to v}) \geq {A_t(G_{u\to v})x^t} \geq {A_t(G)x^t}= \rho_t(G).
			        \end{align}
                The equality in the first inequality holds if and only if $x$ is the eigenvector corresponding to the $t$-clique spectral radius $\rho_t(G_{u\to v})$ of $G_{u\to v}$, and the equality in the second inequality holds only if $x_u = x_v$ or $G = G_{u \to v}$.

                If the graph $G$ is not $t$-clique connected, then by Lemma~\ref{2L2-1}, there exists a $t$-clique connected component $H$ in $G$ such that $\rho_t(H) = \rho_t(G)$. We now discuss different cases based on the selection of vertices $u$ and $v$.

                (1) If $u, v \in V(H)$, let $x$ be the $t$-clique Perron vector of graph $H$, with the components corresponding to vertices $u$ and $v$ denoted by $x_u$ and $x_v$, respectively. Without loss of generality, we may assume $x_u \leq x_v$. Since $H_{u \to v}$ is a subgraph of $G_{u \to v}$, it follows that
                    \begin{align*}
                         \rho_t(G_{u \to v}) \geq \rho_t(H_{u \to v}) \geq {A_t(H_{u \to v})x^t} \geq {A_t(H)x^t} = \rho_t(H)= \rho_t(G).
			        \end{align*}

                (2) If $u \in V(G) \setminus V(H)$(for the case where $v  \in V(G) \setminus V(H)$, we can similarly consider $G_{v \to u}$), it is straightforward to verify that $H$ is a subgraph of $G_{u \to v}$, and thus we have
                    \begin{align*}
                         \rho_t(G_{u \to v}) \geq \rho_t(H) = \rho_t(G).
			        \end{align*}

        \end{proof}

        Let $\mathrm{SPEX}_t(n,C_{\geq k})$ denote the family of graphs with $n$ vertices that do not contain any cycle of length at least $k$ and attain the maximum $t$-cliques spectral radius.

        Define the set
            \begin{align*}
                &\mathrm{SPEX}_t'(n,C_{\geq k})\\=&\{G|G \in \mathrm{SPEX}_t(n,C_{\geq k})\ and\ for \ any\ \{u,v\} \in E(G),\ P_u(v)=\emptyset\ or\ P_v(u)=\emptyset\}.
            \end{align*}
       By Lemmas~\ref{Gao1} and \ref{CL1}, for any $G \in \mathrm{SPEX}_t(n, C_{\geq k})$, we can apply a sequence of operations to $G$ to obtain a graph $G'$ such that $G' \in \mathrm{SPEX}_t'(n, C_{\geq k})$, which implies that $\mathrm{SPEX}_t'(n, C_{\geq k})$ is nonempty. The following Lemma presents a property of graphs in the set $\mathrm{SPEX}_t'(n,C_{\geq k})$.

        \begin{lem}\label{CL2}
                 Let $\bar{G} \in \mathrm{SPEX}_t'(n,C_{\geq k})$, then every connected induced subgraph $G$ of $\bar{G}$ has diameter at most 2. Moreover, there exists at least one vertex $i$ in $G$ satisfying $d_{G}(i)=|V(G)|-1$.

        \end{lem}
        \begin{proof}

                Let $\bar{G} \in \mathrm{SPEX}_t'(n,C_{\geq k})$, and let $G$ be any connected induced subgraph of $\bar{G}$, observe that for any adjacent vertices $u$ and $v$ in $G$, either $P_u(v)$ or $P_v(u)$ is empty in $G$.

                If the diameter of $G$ is greater than $2$, then there exists a pair of non-adjacent vertices $u_1$ and $v_1$ such that the distance between them is $3$, i.e., there exist $u_2, v_2$ with $u_1 \sim u_2 \sim v_2 \sim v_1$ and $u_1 \nsim v_2$, $v_1 \nsim u_2$. Consequently, $P_{u_2}(v_2)\neq \emptyset$ and $P_{v_2}(u_2)\neq \emptyset$ in $G$, leading to a contradiction.

                Next, we prove that there exists a vertex $i \in V(G)$ with $d_{G}(i)=|V(G)|-1$. Fix an arbitrary vertex $u \in V(G)$, if $d_{G}(u)=|V(G)|-1$, the proof is complete. Otherwise, we discuss the case where $d_{G}(u) < |V(G)|-1$. We make the following definition in the graph $G$, let
                    \begin{align*}
                         Y(v)=& V(G) \setminus (N_{G}(v) \cup \{v\}), S(v) = N_{G}(v)\cap N(Y(v)), T(v) =& N_{G}(v)\setminus S(v),
			        \end{align*}
                where $N(Y(v))=\{u|u \in N_{G}(w), \text{ for some } \ w \in Y(v)\}$.
                If $d_{G}(u) < |V(G)|-1$, then $Y(u) \neq \emptyset$, since $G$ is connected, $S(u) \neq \emptyset$. For any $v_1 \in S(u)$ and any $v_2 \in N_{G}(u)\setminus\{v_1\}$, $v_1 \sim v_2$ hold, otherwise, there exists $w \in Y(u)\cap N_{G}(v_1)$ such that $v_2 \in P_{v_1}(u) \neq \emptyset$ and $w \in P_u(v_1) \neq \emptyset$, which is a contradiction.
                Since that $G$ has diameter $2$, $Y(u) \subseteq \bigcup_{v \in S(u)} N_{G}(v)$. If $|S(u)|=1$, let $S(u)=\{v\}$, then $Y(u) \subseteq N_{G}(v)$, and thus $d_{G}(v)=|G|-1$, completing the proof. If $|S(u)| \neq 1$, take any $v_1, v_2 \in S(u)$, since $v_1 \sim v_2$, either $P_{v_1}(v_2)=\emptyset$ or $P_{v_2}(v_1)=\emptyset$. Without loss of generality, assume $P_{v_2}(v_1)=\emptyset$. Then,
                    \begin{align*}
                        (N_{G}(v_1) \cap Y(u)) \subseteq (N_{G}(v_2) \cap Y(u)),
			        \end{align*}
                 Meaning there exists a vertex $v'$ such that for all $v \in S(u)$,
                    \begin{align*}
                        N_{G}(v) \cap Y(u) \subseteq N_{G}(v') \cap Y(u).
			        \end{align*}
                Since $Y(u) \subseteq \bigcup_{v \in S(u)} N_{G}(v)$, it follows that $Y(u) \subseteq N_{G}(v')$, and therefore $d_{G}(v')=|V(G)|-1$, completing the proof.

        \end{proof}

        In a graph $G$, a vertex of degree $|V(G)| - 1$ is called a universal vertex of $G$, the set of all universal vertices in $G$ is called the universal vertex set of $G$, denoted by $Un(G)$, and we write $un(G) = |Un(G)|$. It is easy to see that $Un(G)$ is a clique in $G$.

        Now we use Lemma~\ref{CL2} to analyze the structure of graphs in $\mathrm{SPEX}_t'(n, C_{\geq k})$. Let $G \in \mathrm{SPEX}_t'(n, C_{\geq k})$. Select an arbitrary connected component of $G$, denoted by $G_1$. By Lemma~\ref{CL2}, we know that $Un(G_1) \neq \emptyset$. If $V(G_1) \setminus Un(G_1) \neq \emptyset$, let $G_1'$ be the subgraph of $G_1$ induced by the vertex set $V(G_1) \setminus Un(G_1)$. Choose an arbitrary connected component of $G_1'$, denoted by $G_2$, and repeat this process.

        It is evident that for any $G_m$, the set $\bigcup_{i=1}^m Un(G_i)$ forms a clique in $G$ with $\sum_{i=1}^m un(G_i)$ vertices. Since $G$ is $C_{\geq k}$-free, we must have
            \begin{align*}
                k > \sum_{i=1}^m un(G_i) \geq m,
	       \end{align*}
        therefore, this process must terminate in at most $k - 1$ steps for any sequence of choices. That is, there exists some $m < k$ such that $V(G_m) \setminus Un(G_m) = \emptyset$, holds for any choice of connected components.

        For any positive integer $i$, we define the connected components of the subgraph induced on $V(G_i) \setminus Un(G_i)$ as the subordinate component of $G_i$. A graph $G_i$ satisfying $V(G_i) \setminus Un(G_i) = \emptyset$ is called a level-1 component, and it is easy to see that any level-1 component is a clique. If all subordinate components of $G_i$ are level-1 components, then $G_i$ is called a level-2 component. If every subordinate component of $G_i$ is either a level-1 or level-2 component, then $G_i$ is called a level-3 component.

        Next, we continue to study the structure of graphs in $\mathrm{SPEX}_t'(n, C_{\geq k})$. Before that, we first introduce a notation. For a subgraph $H$ of a graph $G$ and a $n$-order vertor $x $, define
            \begin{align*}
                L_r(G, H, x) = \sum_{\{i_1,\cdots ,i_r\}\in K_r(G)\ and\ \{i_1,\cdots ,i_r\}\subseteq N_c(H)\setminus V(H)}x_{i_1}\cdots x_{i_r}.
            \end{align*}
        where $N_c(H)$ is the set of common neighbors of the vertex set $V(H) $ in $G $, i.e., $N_c(H) = \{ i \in V(G) \mid i \in \bigcap_{j \in V(H)} N_G(j) \}$.

        \begin{lem}\label{CL4}
                 Let $\bar{G} \in \mathrm{SPEX}_t'(n, C_{\geq k})$, $G $ is a $t$-clique connected compoment of $\bar{G} $ that satisfies $\rho_t(G) = \rho_t(\bar{G})$, then the graph $G$ does not contain any level-3 components.
        \end{lem}

        \begin{proof}

                Suppose $G$ contains a level-3 component $H$, the level-2 components among the subordinate components of $H$ be ordered by the size of their universal vertex sets in descending order as $D_1, \cdots, D_p$. We now consider the level-2 components on $H$.
            \begin{claim}\label{CC1}
                For each $h \in [p]$, the level-2 component $D_h$ has at most one subordinate component that contains a clique with more than one vertex.
            \end{claim}
            \begin{proof}
                Let $un(D_h) = m_0$. By definition, all subordinate components of the level-2 component $D_h$ are complete graphs, denote them in descending order of size as $M_1, M_2, \cdots, M_l$, and let their respective numbers of vertices be $m_1, m_2, \cdots, m_l$.
                By symmetry, for the $t$-clique Perron vector $x$ of $G$, any two vertices $i, j$ such that both belong to the same clique $M_k$, or both belong to $Un(D_h)$, have the same $t$-clique Perron vector component, hence, we denote the this $t$-clique Perron vector entry corresponding to any vertex in $M_k$ by $x_k$ and corresponding to any vertex in $Un(D_h)$ by $x_0$.
                For given $i, j \geq 1$, if $m_i > m_j$, then $x_i > x_j$, otherwise, if $m_i > m_j$ and $x_i \leq x_j$, we arbitrarily select $m_j$ vertices from $V(M_i)$ to form a induced subgraph of $M_i$, denoted as $M_i'$, then, by swapping the components of $x$ corresponding to $V(M_i')$ and $V(M_j)$, we obtain the vector $y$ with its $t$-norm equal to $1$. According to the definition, $L_r( G , M_i, x)=L_r( G ,M_j, x)$ holds for any $r$, thus, we have
                    \begin{align*}
                        &A_t(G)y^t-A_t(G)x^t
                        \\= &t\sum_{r=1}^t L_{t-r}(G, M_i, x) \left( \sum_{s=1}^r \begin{pmatrix} m_j\\s \end{pmatrix} \begin{pmatrix}m_i-m_j\\ r-s \end{pmatrix} x_i^{r-s}x_j^{s} + \begin{pmatrix}m_j\\r\end{pmatrix} x_i^r \right.
                        \\& \left.- \sum_{s=1}^r \begin{pmatrix} m_j\\s \end{pmatrix} \begin{pmatrix}m_i-m_j\\ r-s \end{pmatrix} x_i^{r} - \begin{pmatrix}m_j\\r\end{pmatrix} x_j^r \right)
                        \\= &t\sum_{r=1}^t L_{t-r}(G, M_i, x) \sum_{s=1}^{r-1} \begin{pmatrix} m_j\\s \end{pmatrix} \begin{pmatrix}m_i-m_j\\ r-s \end{pmatrix} x_i^{r-s}(x_j^{r}-x_i^{r})
                        \\\geq &0,
                    \end{align*}
                equality holds if and only if $x_i=x_j$, but since $m_i > m_j$, by comparing the components of the vertices in $M_i $ and $M_j $ through Equation (1), it can be concluded this cannot be true, therefore, we have $A_t(G)y^t-A_t(G)x^t > 0$, this contradicts the fact that $x$ is the $t$-clique Perron vector.

                If $m_2 = 1$, then the Claim already holds, so it suffices to consider the case where $m_2 \geq 2$. Suppose $m_2 \geq 2$, consider the following operation on $D_h$ to obtain a new graph $D_h'$:

                \noindent (1) Arbitrarily choose $m_2 - 1$ vertices from $V(M_1)$ to form an induced subgraph of $M_1$, denoted by $M_1'$. Connect each pair $(i, j)$, where $i \in V(M_1')$ and $j \in \bigcup_{i=2}^l V(M_i)$;

                \noindent (2) For each $i \in {2, \cdots, l}$, delete all edges in $E(M_i)$.

                It is easy to see that $D_h'$ is still a level-2 component, and it has at most one subordinate component whose clique size is greater than 1. Define a new graph $G'$ with vertex set $V(G') = V(G)$ and edge set $E(G') = E(G) \setminus E(D_h) \cup E(D_h')$, then $p(D_h') = 2m_0+m_1+m_2-1 \leq \sum_{i=0}^{m_0+1}m_i = p(D_h)$, furthermore, $c(G') \leq c(G) \leq c(\bar{G})$, and
                    \begin{align*}
                        &\rho_t(G')-\rho_t(G)
                        \\\geq & A_t(G')x^t-A_t(G)x^t
                        \\= &\sum_{r=2}^t L_{t-r}(G, D_h, x) \left( \sum_{i=2}^l m_i \sum_{s=1}^{r-1} \begin{pmatrix}m_2-1\\s\end{pmatrix} \begin{pmatrix}m_0\\r-s-1\end{pmatrix}x_1^sx_0^{r-s-1}x_i \right)\\
                        &- \sum_{r=2}^t L_{t-r}(G, D_h, x) \left( \sum_{i=2}^l \sum_{s=2}^{r} \begin{pmatrix}m_i\\s\end{pmatrix} \begin{pmatrix}m_0\\r-s\end{pmatrix}x_i^{s}x_0^{r-s} \right)
                        \\= &\sum_{r=2}^t L_{t-r}(G, D_h, x) \\
                        &\sum_{i=2}^l \sum_{s=1}^{r-1}\begin{pmatrix}m_0\\r-s-1\end{pmatrix}x_0^{r-s-1}x_i\left( \begin{pmatrix}m_2-1\\s\end{pmatrix}m_ix_1^s-\begin{pmatrix}m_i\\s+1\end{pmatrix}x_i^s \right)
                        \\\geq &0,
                    \end{align*}

                The equality in the first inequality holds if and only if $G'$ and $G$ share the same $t$-clique Perron vector.
                By symmetry, in the graph $G'$, the vertices in $M_1'$ and those in $M_0$ have the same $t$-clique Perron vector components,
                however, in the graph $G$, by comparing the components of the vertices in $M_1 $ and $M_0 $ through Equation (1), we have $x_0 > x_1$, which implies that the $t$-clique Perron vectors of $G$ and $G'$ are different, therefore, the equality does not hold. The second inequality holds because $m_2 \geq m_i$ and $x_1 \geq x_i$ for all $i \in {2, 3, \cdots, l}$, hence,

                    \begin{align*}
                        \rho_t(G') > \rho_t(G) = \rho_t(\bar{G}),
                    \end{align*}
                this is a contradiction.
            \end{proof}

            \begin{claim}\label{CC2}
                $H$ have exactly one level-2 subordinate component.
            \end{claim}

            \begin{proof}

                Let $D_1$ and $D_2$ be two different level-2 components which are subordinate components of $H$. Denote the subordinate components of $D_1$, ordered by clique size from largest to smallest, as $M_1^1, M_1^2, \ldots, M_1^{p_1}$, and those of $D_2$ as $M_2^1, M_2^2, \ldots, M_2^{p_2}$.
                Obviously $un(D_i) \leq p_i - 1$; otherwise, by connecting any two vertices in $D_i$ to obtain a  complete graph $D_i'$, the circumference of $G$ does not increase by $p(D_i)=p(D_i')=|V(D_i)|$, and the $t$-clique spectral radius increases, because the operation introduces new $t$-cliques, for instance, for $u \in Un(D_i)$ and $v, w \in V(D_i) \setminus Un(D_i)$ with $v \nsim w$, since $G$ is $t$-clique connected, it follows from the neighborhoods of these vertices that there exists a $t$-clique $K$ in $G$ containing $u$ and $v$, then, connecting $w$ and $v$ yields a new $t$-clique $K'$, where $V(K') = V(K) \setminus \{u\} \cup \{w\}$.

                By Claim~\ref{CC1}, each level-2 component has at most one subordinate component with more than one vertex, therefore, $|M_i^{j_i}| = 1$, where $j_i \in \{2,3,\ldots,p_i\}$ and $i \in \{1,2\}$.

                Let $x$ be the $t$-clique Perron vector of $G$, by symmetry, we may assume that the entry of $x$ corresponding to any vertex in $Un(D_1)$ is $x_1$, and that corresponding to any vertex in $Un(D_2)$ is $x_2$, let the entries corresponding to $M_1^2, \ldots, M_1^{p_1}$ be denoted by $x_1'$, and those corresponding to $M_2^2, \ldots, M_2^{p_2}$ by $x_2'$.

                Now consider modifying the graph $G$ as follows: delete all edges with one incident vertex in $\bigcup_{i=2}^{p_2} V(M_2^i)$ and the other in $Un(D_2)$, and then add all possible edges between vertices in $\bigcup_{i=2}^{p_2} V(M_2^i)$ and $Un(D_1)$ to obtain a new graph $G'$. Similarly, delete all edges with one incident vertex in $\bigcup_{i=2}^{p_1} V(M_1^i)$ and the other in $Un(D_1)$, and then add all possible edges between vertices in $\bigcup_{i=2}^{p_1} V(M_1^i)$ and $Un(D_2)$ to obtain a new graph $G''$, clearly, $c(G') \leq c(G) \leq c(\bar{G})$ and $c(G'') \leq c(G) \leq c(\bar{G})$. It is easy to see that $L_r(G, D_1, x) = L_r(G, D_2, x)$ holds for any $r$, hence
                    \begin{align*}
                        &\rho_t(G')-\rho_t(G)
                        \geq  A_t(G')x^t-A_t(G)x^t
                        \\= & t\sum_{r=1}^{t}L_{t-r}(G, D_1, x) (p_2-1)x_2' \left( \begin{pmatrix}un(D_1)\\r-1\end{pmatrix}x_1^{r-1}-\begin{pmatrix}un(D_2)\\r-1\end{pmatrix}x_2^{r-1} \right),
                    \end{align*}
                and
                    \begin{align*}
                        &\rho_t(G'')-\rho_t(G)
                        \geq  A_t(G'')x^t-A_t(G)x^t
                        \\= & t\sum_{r=1}^{t}L_{t-r}(G, D_1, x) (p_1-1)x_1' \left( \begin{pmatrix}un(D_2)\\r-1\end{pmatrix}x_2^{r-1}-\begin{pmatrix}un(D_1)\\r-1\end{pmatrix}x_1^{r-1} \right),
                    \end{align*}
                By comparing the two equations above, it is clear that either $\rho_t(G')-\rho_t(G)\geq 0$ or $\rho_t(G'')-\rho_t(G)\geq 0$, without loss of generality, assume $\rho_t(G')-\rho_t(G)\geq 0$, where equality holds if and only if the graph $G'$ has the same $t$-clique Perron vector as $G$ and
                    \begin{align*}
                        \begin{pmatrix}un(D_1)\\r-1\end{pmatrix}x_1^{r-1}= \begin{pmatrix}un(D_2)\\r-1\end{pmatrix}x_2^{r-1},
			        \end{align*}
                however, from the characteristic equation corresponding to any vertex in ${Un}(D_1)$ in Equation (1), we conclude that $x$ cannot be an eigenvector of $G'$, thus the equality condition cannot hold. Furthermore, since $G'$ has strictly fewer level-2 subordinate components than $G$, and $\rho_t(G') > \rho_t(G) =\rho_t(\bar{G})$, we arrive at a contradiction.

            \end{proof}

            We now proceed to conclude the proof of the lemma. By Claim~\ref{CC2},
            the subordinate component set of $H$ contains exactly one level-2 component, i.e., $D_1$.
            With the exception of $D_1$, all other subordinate components of $H$ are cliques, ordered by size in descending sequence as $M_1, M_2, \ldots, M_l$, with their respective sizes denoted by $m_1, \ldots, m_l$. Obviously $l \neq 0$, otherwise $H$ would be a level-2 component, which is a contradiction.

            If $m_1 = 1$, then for $i \in [l]$, $M_i$ contains only one vertex, after fully connecting these vertices with the vertices in $Un(D_1)$, the graphs $G'$ and $H'$ are obtained from the graph $G$ and its subgraph $H$, respectively. Then, compared to $G$, $G'$ generates new $t$-cliques and $p(H') = p(H)$, so $\rho_t(G') > \rho_t(G)$ and $c(G') \leq c(G)$.

            If $l = 1$, by adding all edges between the vertex sets $V(Un(D_1))$ and $V(M_1)$, we immediately obtain graph $G'$ with $\rho_t(G') > \rho_t(G)$ and $c(G')\leq c(G)$.

            Therefore, it suffices to consider the case when $m_1 \geq 2$ and $l \geq 2$.

            When $m_1 \geq 2$ and $l \geq 2$, select a maximum clique on $D_1$ and denote it as $K(D_1)$, let $|K(D_1)| = k(D_1)$. It is easy to see that $Un(D_1) \subsetneq K(D_1)$.

            If $un(D_1) > m_1$, similarly to Claim~\ref{CC1}, it can be shown that the following operation will yield a graph $G'$ such that $c(G') \leq c(G)$ and $\rho_t(G') > \rho_t(G)$:

            \noindent (1) Arbitrarily choose $m_1 - 1$ vertices from $V(Un(D_1))$ to form an induced subgraph of $Un(D_1)$, denoted by $Un(D_1)'$. Connect each pair $(i, j)$, where $i \in V(Un(D_1)')$ and $j \in \bigcup_{i=1}^l V(M_i)$;

            \noindent (2) For each $i \in {1, \cdots, l}$, delete all edges in $E(M_i)$.

            If $k(D_1) \geq m_1 \geq un(D_1)$, consider the following method to move the neighbors of the vertices in $M_i$ to the vertices of $K(D_1)$. Denote the vertices in $M_1$ as $v_1^1, v_1^2, \cdots, v_1^{m_1}$.
            Select $m_1 - 1$ vertices on $K(D_1)$ (preferentially choosing vertices from $Un(D_1)$), denoted as $u_1^1, u_1^2, \cdots, u_1^{m_1 - 1}$. Let $H_1^0 = H$. For $i = 1, \cdots, m_1 - 1$, let $H_1^i = {H_1^{i-1}}_{v_1^i \to u_1^i}$, obtaining the graph $H_1^{m_1 - 1}$, then, the induced subgraph of $H_1^{m_1 - 1}$ on $V(Un(H)) \cup V(D_1) \cup V(M_1)$ is a proper subgraph of $S_{|V(H)|, m_1+un(H)-1, k(D_1) - m_1 + 1}$.
            Next, denote the vertices in $M_2$ as $v_2^1, v_2^2, \cdots, v_2^{m_2}$. Let $H_2^0 = H_1^{m_1 - 1}$, for $i = 1, \cdots, m_2 - 1$, let $H_2^i = {H_2^{i-1}}{v_2^i \to u_1^i}$, obtaining the graph $H_2^{m_2 - 1}$.
            Continue this process iteratively for $i = 3, \ldots, l$, applying the same operation to each $M_i$, to finally obtain the graph $H_l^{m_l - 1}$.
            Then, $H_l^{m_l - 1}$ is still a proper subgraph of
                \begin{align*}
                        S_{|V(H)|, m_1+un(H)-1, k(D_1) - m_1 + 1}.
                \end{align*}
            It is easy to see that $p(S_{|V(H)|, m_1+un(H)-1, k(D_1) - m_1 + 1}) = 2un(H)+m_1+k(D_1)-1 \leq p(H)$.
            Furthermore, by replacing the edge subset $E(H)$ in $G$ with
                \begin{align*}
                        E(S_{|V(H)|, m_1+un(H)-1, k(D_1) - m_1 + 1})
                \end{align*}
            we obtain the graph $G'$, we have $c(G') \leq c(G)$. And then, according to Lemma~\ref{CL1}, there is $\rho_t(G') > \rho_t(G)$.

            If $m_1 > k(D_1)$, similar to the case when $k(D_1) \geq m_1 \geq un(D_1)$, we can likewise obtain a graph $G'$ such that $c(G') \leq c(G)$ and $\rho_t(G') > \rho_t(G)$ by sequentially transferring the neighbors of the vertices in $K(D_1)$ and $M_i$ to the vertices in $M_1$ for all $i = {2, 3, \cdots, l}$.

            That is, there exists a $C_{\geq k}$-free graph $G'$ with $n$ vertices such that
                    \begin{align*}
                        \rho_t(G')>\rho_t(G) = \rho_t(\bar{G}),
                    \end{align*}
            This contradicts the assumption that $\bar{G}\in \mathrm{SPEX}_t'(n,C_{\geq k})$.

        \end{proof}

            By Lemma~\ref{CL4}, if $\bar{G} \in \mathrm{SPEX}_t'(n, C_{\geq k})$, $G $ is a $t$-clique connected compoment of $\bar{G} $ that satisfies $\rho_t(G) = \rho_t(\bar{G})$, then the graph $G$ is a level-2 component or complete graph, furthermore, if $G$ is a level-2 component,  Claim~\ref{CC1} implies that among its subordinate components, at most one clique has size greater than $1$. Thus, $G$ can be expressed as $S_{n',un(G),l}$, i.e., $K_{un(G)} \vee \{K_l \cup I_{n'-un(G)-l}\}$, where $n'$ represents the number of vertices in graph $G$, and $l \neq 1$ is a integer. When $G $ is a complete graph, we have $un(G)=n'$, $l = 0$ and $c(G) = n' < k$, since $G$ is $t$-clique connected, $n' \geq t$;
            when $G $ is a level-2 component, we have $un(G) < n'$ and $l \geq 2$, furthermore, we have $n'-un(G)-l > un(G)$, otherwise, by fully connecting all pairs of vertices in graph $G$, the complete graph $K_{n'} $ is obtained, then we have $c(K_{n'}) = c(G) = n'$ and $\rho_t(K_{n'}) > \rho_t(G) = \rho_t(\bar{G}) $, this is contradictory.
            Then $c(G) = 2un(G) + l - 1 < k$, that is $\frac{k}{2} > {un}(G)$. And then, since $G$ is $t$-clique connected, $un(G) \geq t-1$.

            \begin{proof}[Proof of Theorem~\ref{CT1-1}]
            When $ t > \lfloor \frac{k+1}{2} \rfloor $, for $\bar{G} \in \mathrm{SPEX}_t'(n, C_{\geq k})$, $G $ is a $t$-clique connected compoment of $\bar{G} $ that satisfies $\rho_t(G) = \rho_t(\bar{G})$, if $G $ is a level-2 component, we have $ t > \lfloor \frac{k+1}{2} \rfloor \geq \frac{k}{2} > {un}(G) $, since $ t $, $ \lfloor \frac{k+1}{2} \rfloor $, and $ {un}(G) $ are all integers, it follows that $ t-1 > {un}(G) $, this is contradictory. Therefore, $ G $ is a complete graph, hence, $ G = K_{k-1} $ by Lemma~\ref{2L4}.
            This completes the proof.
            \end{proof}

             Now we prove that $\mathrm{SPEX}_t'(n,C_{\geq k}) $ contains unique elements when $2 \leq t \leq \lfloor \frac{k+1}{2} \rfloor$.

        \begin{lem}\label{CL5}
            For sufficiently large $n$, $2 \leq t \leq \lfloor \frac{k+1}{2} \rfloor$, when $k $ is odd (resp. even), if an $n$-vertex graph $\bar{G} \in \mathrm{SPEX}_t'(n,C_{\geq k}) $, then $\bar{G} = S_{n, \frac{k-1}{2} }$ (resp. $\bar{G} = S_{n, \frac{k-2}{2} }^+$).
        \end{lem}

        \begin{proof}

            Suppose the graph $\bar{G} \in \mathrm{SPEX}_t'(n,C_{\geq k}) $, given the graph $G = S_{n',un(G),l} $ as defined above is a $t$-clique connected compoment of $\bar{G} $ that satisfies $\rho_t(G) = \rho_t(\bar{G})$.

            \begin{claim}\label{CC4}
                $G$ is a level-2 component.
            \end{claim}

            \begin{proof}
                If $G$ is not a level-2 component, then $G = K_{k-1} $. By Lemma~\ref{1L1}, using the method of Lagrange multipliers and incorporating the symmetry of graph $K_{k-1}$, we can calculate that
                    \begin{align*}
                        \rho_t(K_{k-1}) = \max\{A_t(K_{k-1})x^t \mid \sum_{i=1}^{n} x_i^t = 1\} = \frac{t}{k-1}\begin{pmatrix}k-1\\t\end{pmatrix},
                    \end{align*}

                It is easy to see that the graph $S_{n,\lfloor \frac{k-1}{2} \rfloor}$ is $C_{\geq k}$-free. When $t \leq \lfloor \frac{k+1}{2} \rfloor$, By incorporating the symmetry of the graph $S_{n,\lfloor \frac{k-1}{2} \rfloor}$, based on the characteristic equation and incorporating the symmetry, we can derive the following result.
                    \begin{align*}
                        \rho_t(S_{n,\lfloor \frac{k-1}{2} \rfloor}) = \Theta(n^{\frac{t-1}{t}}),
                    \end{align*}
                then, for $k > t \geq 2$, when $n$ is sufficiently large, we have
                    \begin{align*}
                        \rho_t(S_{n,\lfloor \frac{k-1}{2} \rfloor}) > \rho_t(K_{k-1}),
                    \end{align*}
                this contradicts the assumption that $G \in \mathrm{SPEX}_t'(n, C_{\geq k})$.
            \end{proof}

            \begin{claim}\label{CC5}
                $\bar{G} = G$.
            \end{claim}

            \begin{proof}
                If $\bar{G} \neq G$, then $n' \neq n$, meaning there exists a vertex $i \in V(\bar{G}) \setminus V(G)$. Remove all edges incident to vertex $i$, and connect $i$ to $j$, for every $j \in un(G)$, resulting in a new graph $G'$. Due to $n'-un(G)-l \geq un(G) $, we have $c(G') \leq c(G) \leq c(\bar{G})$, and by Lemma~\ref{2L4}, we have $\rho_t(G') > \rho_t(G) = \rho_t(\bar{G})$, which contradicts the assumption that $\bar{G} \in \mathrm{SPEX}_t'(n, C_{\geq k})$. That is $\bar{G} = G = S_{n,{un}(G),l}$.
            \end{proof}

             We now consider the sizes of $un(G)$ and $l$. In the graph $S_{n, a, b}$ (with $a \geq 1$, $b \geq 3$), let $A$, $B$, and $C$ be the vertex sets of $K_a$, $K_b$, and $I_{n-a-b}$ of $S_{n,a,b}$, respectively. Lemma~\ref{2L3} shows its $t$-clique tensor has a unique positive $t$-clique Perron vector $x $ satisfying, where by symmetry $x_i=x_j $ holds for any $i$ and $j$ that are in the same set among $A $, $B $ or $C $, therefore we can let $x_a $, $x_b $ and $x_c $ denote the $t$-clique Perron vector components for vertices in $A$, $B$ and $C$ respectively, from the characteristic equation, it is straightforward to see that $x_a > x_b > x_c$.
             Due to $n $ being sufficiently large and $b \geq 3 $, we immediately have $c(S_{n,a+1,b-2}) = c(S_{n,a,b}) $, the comparison of $t$-clique spectral radii between these two graphs follows
                \begin{align*}
                    &\rho_t(S_{n,a+1,b-2})-\rho_t(S_{n,a,b})
                    \\\geq & (n-a-b)\begin{pmatrix} a\\t-2 \end{pmatrix}x_a^{t-2}x_bx_c-\sum_{i=1}^{t-1}\begin{pmatrix} b-2\\i \end{pmatrix}\begin{pmatrix} a\\t-i-1 \end{pmatrix}x_a^{t-i-1}x_b^{i+1}
                    \\\geq & x_bx_a^{t-1}\left((n-a-b)\begin{pmatrix} a\\t-2 \end{pmatrix}\frac{x_c}{x_a}-\begin{pmatrix} a+b-2\\t-1 \end{pmatrix} \right).
                \end{align*}
            According to the characteristic equation, it can be inferred that
                \begin{align*}
                    \rho_t(S_{n,a,b})x_c^{t-1}=t\begin{pmatrix} a\\t-1 \end{pmatrix}x_a^{t-1}.
                \end{align*}
            Based on the characteristic equation and incorporating the symmetry, we can derive the following result.
                \begin{align*}
                    \rho_t(S_{n,a,b})=\Theta(n^{\frac{t-1}{t}}),
                \end{align*}
            then, there is
                \begin{align*}
                    \frac{x_c}{x_a}=\Theta(n^{-\frac{1}{t}}),
                \end{align*}
            so, when $n $ is sufficiently large
                \begin{align*}
                    (n-a-b)\begin{pmatrix} a\\t-2 \end{pmatrix}\frac{x_a}{x_c} > \begin{pmatrix} a+b-2\\t-1 \end{pmatrix},
                \end{align*}
            this means that
                \begin{align*}
                    &\rho_t(S_{n,a+1,b-2})-\rho_t(S_{n,a,b}) > 0.
                \end{align*}
            That is, for a given circumference $k$, under the constraint of $2a+b-1 < k$, when $k$ is odd, $t$-clique spectral radii of $S_{n,a,b}$ reaches its maximum at $a= \frac{k-1}{2} $ and $b=1$; when $k$ is even, it reaches maximum at $a= \frac{k-2}{2} $ and $b=2$.

        \end{proof}

        For sufficiently large $n$, $2 \leq t \leq \lfloor \frac{k+1}{2} \rfloor$, Lemma~\ref{CL5} establishes that for odd $k$ (resp. even $k$), the graph $S_{n, \frac{k-1}{2} }$ (resp. $S_{n, \frac{k-2}{2} }^+$) not only attains the maximum spectral radius in $\mathrm{SPEX}_t(n,C_{\geq k})$ but is further the unique extremal graph in $\mathrm{SPEX}_t'(n,C_{\geq k})$, namely $\mathrm{SPEX}_t'(n,C_{\geq k})=\{S_{n, \frac{k-1}{2} }\}$ (resp. $\{S_{n, \frac{k-2}{2} }^+\}$).We now prove the uniqueness of extremal graphs in $\mathrm{SPEX}_t(n,C_{\geq k})$.

        \begin{proof}[Proof of Theorem~\ref{CT1}]

            If $k$ is odd (resp. even), by the definition of $\mathrm{SPEX}_t'(n,C_{\geq k})$, if $\mathrm{SPEX}_t(n,C_{\geq k})\setminus\mathrm{SPEX}_t'(n,C_{\geq k})\neq \emptyset$, then there exists $G \in \mathrm{SPEX}_t(n,C_{\geq k})\setminus\mathrm{SPEX}_t'(n,C_{\geq k})$ such that $G$ contains an adjacent pair $u,v$ with $G_{u \to v} = S_{n, \frac{k-1}{2} }$ (resp. $S_{n, \frac{k-2}{2} }^+$).
            By definition of $G_{u \to v}$, we have $v \in {Un}(S_{n, \frac{k-1}{2} })$ (resp. $Un(S_{n, \frac{k-2}{2} }^+)$) and $u \notin {Un}(S_{n, \frac{k-1}{2} })$ (resp. $Un(S_{n, \frac{k-2}{2} }^+)$).
            Clearly $G$ remains $t$-clique connected. Since $u$ and $v$ have distinct $t$-clique Perron vector components in $S_{n, \frac{k-1}{2} }$ (resp. $S_{n, \frac{k-2}{2} }^+$) , the equality condition in formula~\ref{align1} fails, i.e.,
                \begin{align*}
                    \rho_t(S_{n, \frac{k-1}{2} }) > \rho_t(G)\ (resp. \rho_t(S_{n, \frac{k-2}{2} }^+)> \rho_t(G)).
                \end{align*}
            contradicting $G$'s extremality, then $\mathrm{SPEX}_t(n,C_{\geq k})\setminus \mathrm{SPEX}_t'(n,C_{\geq k}) = \emptyset$. This completes the proof.
        \end{proof}

\section{Proof of Theorems~\ref{CT3-1} and \ref{CT3}}

        Let $\mathrm{SPEX}_t(n,P_{k})$ denote the family of graphs with $n$ vertices that which does not contain $P_k $ and attain the maximum $t$-cliques spectral radius. Denote
                \begin{align*}
                    &\mathrm{SPEX}_t'(n,P_{ k})\\=&\{G|G \in \mathrm{SPEX}_t(n,P_{ k})\mathrm{\ and\ for\ any\ }\{u,v\} \in E(G),\ P_u(v)=\emptyset\ or\ P_v(u)=\emptyset\}.
                \end{align*}
         By Lemmas~\ref{L4-1} and \ref{CL1}, for any $G \in \mathrm{SPEX}_t(n, P_k)$, we can apply a sequence of operations to $G$ to obtain a graph $G'$ such that $G' \in \mathrm{SPEX}_t'(n, P_k)$, which implies that $\mathrm{SPEX}_t'(n, P_k)$ is nonempty. Similar to the proof in the third section, we can derive Theorems~\ref{CT3-1} and \ref{CT3} in parallel.

\vspace{3mm}

\noindent
\textbf{}
\vspace{3mm}
\noindent

\section*{}

\end{document}